\documentclass{article}
\usepackage{amsmath}
\usepackage{amssymb}
\usepackage{amsmath}
\usepackage{amssymb}
\usepackage{amsthm}
\usepackage{ifthen}

\newcommand{\comment}[1]{} 

\newcommand{\estr}{} 
\newcommand{\lcurl}{\{} 
\newcommand{\rcurl}{\}} 
\newcommand{\Ls}[2][\infty]{\ifthenelse{\equal{#2}{\estr}}{L^{#1}}{L^{#1}(#2)}} 
\newcommand{\pref}[1]{(\ref{#1})}

\newcommand{\ale}[1][\ ]{a.e.#1}

\newcommand{\wrt}[1][ ]{with respect to #1}

\newcounter{spare} 
\newcommand{\srndt}[2][0]{
\setcounter{spare}{#1}\gdef\lf{#2}\ifthenelse{\equal{\lf}{(}}{\gdef\rt{)}}{\ifthenelse{\equal{\lf}{[}}{\gdef\rt{]}}{\ifthenelse{\equal{\lf}{|}}{\gdef\rt{|}}{\ifthenelse{\equal{\lf}{\lcurl}}{\gdef\rt{\rcurl}}{\gdef\rt{.}}}}}\ifthenelse{\value{spare}=0}{\gdef\ldel{\lf}\gdef\rdel{\rt}}{\ifthenelse{\value{spare}=1}{\gdef\ldel{\bigl\lf}\gdef\rdel{\bigr\rt}}{\ifthenelse{\value{spare}=2}{\gdef\ldel{\Bigl\lf}\gdef\rdel{\Bigr\rt}}{\ifthenelse{\value{spare}=3}{\gdef\ldel{\biggl\lf}\gdef\rdel{\biggr\rt}}{\gdef\ldel{\Biggl\lf}\gdef\rdel{\Biggr\rt}}}}}} 

\theoremstyle{plain}
\newtheorem{theorem}{Theorem}[section]
\newtheorem{lemma}[theorem]{Lemma}
 \newtheorem*{defin}{Definition} 

\newtheorem{cor}[theorem]{Corollary}
\newtheorem{prop}[theorem]{Proposition}

\theoremstyle{remark}

\theoremstyle{definition}

\newtheorem{definition}[theorem]{Definition}
\newtheorem{remark}[theorem]{Remark}

\DeclareMathOperator{\ad}{ad}
\DeclareMathOperator{\Ad}{Ad}
\DeclareMathOperator{\Der}{Der}
\DeclareMathOperator{\Ev}{Ev}
\DeclareMathOperator{\GL}{GL}
\DeclareMathOperator{\Gr}{Gr}

\DeclareMathOperator{\Lie}{Lie}

\begin{document}
\title{ INVARIANT RIGID GEOMETRIC STRUCTURES \\
AND SMOOTH PROJECTIVE FACTORS}
\author{Amos Nevo\thanks{Supported by the Institute for Advanced Study, Princeton and ISF grant \# 975-05.  }\\
        Technion IIT\\
        anevo\char64 tx.technion.ac.il \and Robert J. Zimmer
 \\ University of Chicago \\
 president\char64 uchicago.edu}


\maketitle
\begin{abstract} 
 
 \footnotetext{{\it 1991 Mathematics subject classification:}
 22D40, 28D15, 47A35, 57S20, 58E40, 60J50.\newline {\it Key 
 words and phrases}: semi-simple Lie groups, parabolic subgroups,
stationary measure, rigid analytic structure.}

We consider actions of non-compact simple Lie groups preserving an analytic rigid geometric structure of algebraic type on a compact manifold.  The structure is not assumed to be unimodular, so an invariant measure may not exist.  Ergodic stationary measures always exist, and when such a measure has full support, we show the following. 

 1)  Either the manifold admits a smooth equivariant map onto a homogeneous projective variety, defined on an open dense conull invariant set, or the Lie algebra of the Zariski closure of the Gromov representation of the fundamental group contains a Lie subalgebra isomorphic to the Lie algebra of the acting group. As a corollary, a smooth non-trivial homogeneous projective factor does exist whenever the fundamental group of $M$ admits only virtually solvable linear representations, and thus in particular when $M$ is simply connected, regardless of the real rank. 

2)  There exist explicit examples showing that analytic rigid actions of certain simple groups 
(of real rank one) may indeed fail to have a smooth projective factor.  

 3) It is possible to generalize Gromov's theorem on the algebraic hull of the  representation of the fundamental group of the manifold to the case of analytic rigid non-unimodular structures, for actions of simple groups of any real rank. 

An important ingredient in the proofs is a generalization of Gromov's centralizer theorem beyond the case of invariant measures.

 \end{abstract}


\section{Introduction and statements of main results} 
Our purpose here is to study analytic actions of simple Lie groups 
on manifolds which preserve a rigid analytic geometric structure of algebraic type. These structures 
  have been introduced in the  foundational work \cite{g1} as a generalization of Cartan's notion of structures of finite type.  
  The study of rigid structures, both for their own sake and also in connection with rigidity theory for group actions   \cite{zproc} has attracted 
quite a bit  of attention 
(see e.g. \cite{cqb} \cite{feres}\cite{fz}\cite{qb} and references therein). We note however that  
most of the work carried out on group actions preserving rigid structures concentrated on the case 
of unimodular structures. This assumption is used extensively in a number of considerations, particularly in providing an invariant volume form and thus an invariant measure on the manifold, to which considerations related to Borel density theorem and its generalizations can be applied. However, 
such an assumption leaves out some very interesting and natural examples, such as conformal 
pseudo-Riemannian and other conformal structures (see \cite{bn} and \cite{frze} for some recent work in this direction). 
  
Here we consider general, not neceessarily unimodular structures, dispensing with the assumption of an invariant volume form, and assuming only the existence of an ergodic  stationary measure of full support. 
 We thus introduce the following :
 
 \begin{defin}{\bf Rigid analytic stationary systems}. 
  A rigid analytic stationary system $(M,\omega,G,\nu)$ is a compact connected analytic 
 manifold, with rigid analytic geometric structure of algebraic type $\omega$,  a connected non-compact almost simple Lie subgroup  $G\subset  \text{Aut}(M,\omega)$ with finite fundamental group,  and an ergodic probability measure $\nu$ of full support on $M$, which is stationary under an admissible measure $\mu$ on $G$  (we suppress $\mu$ in the notation).
 \end{defin}
 
 Our first main result is then as follows :

\begin{theorem}\label{main}{\bf Fundamental group and the existence of a smooth projective factor.}

Let $(M,\omega,G,\nu)$ be a rigid analytic stationary system. Then one of the following holds :  
 \begin{enumerate}
 \item 
 There exists a generic (i.e., open dense co-null) $G$-invariant subset $M_0 \subseteq M$ 
and a smooth $G$-equivariant 
map $M_0 \rightarrow G/Q$, where $Q$ is a proper parabolic subgroup.
\item The Lie algebra of the Zariski closure of the image of the fundamental group under the Gromov representation contains a Lie subalgebra isomorphic to $\mathfrak g$. 
\end{enumerate}
In particular,  a rigid analytic stationary system admits a non-trivial generic smooth homogeneous projective factor   
if the fundamental group of $M$ admits only virtually solvable linear representations.  
\end{theorem}

As we shall see below, every smooth $G$-manifold with ergodic stationary measure of full support admits a unique maximal projective factor of the form $G/Q$ with the factor map defined and smooth on a generic set.  Thus Theorem \ref{main} shows that this factor is {\it non-trivial} whenever the fundamental group is amenable, or more generally does not contain a free group on two generators. 

As to the necessity of the assumptions in Theorem \ref{main}, we note that rigid analytic stationary systems do not always admit smooth projective factors, even when the stationary measure is not invariant. This is demonstrated by the following :

\begin{theorem}\label{rank1}{\bf Examples for real rank one groups.} Let $G=SO^0(n,1)$ for some $n\ge 2$, and $\mu$ an admissible measure on $G$. Then there exists an analytic action of $G$ on a compact analytic manifold $M$, with a rigid analytic geometric structure of algebraic type $\omega$ preserved by $G$,  such that 
\begin{enumerate}
\item $M$ has a unique $G$-ergodic $\mu$-stationary measure $\nu$, and $\nu$ is of full support on $M$ and  not $G$-invariant  ($\nu$ is in addition $P$-mixing, see \cite{nz1}).  
\item $M$ does not admit a non-trivial generic smooth projective factor. In fact it does not even admit a non-trivial measurable projective factor, for any $G$-quasi-invariant measure. In addition, the $G$-action on $M$ is everywhere locally free.  
\end{enumerate}
\end{theorem}

%


%

 We now turn to a generalization of  Gromov's results regarding rigid analytic structures from the measure-preserving case to the stationary case. We will thus be able to obtain significant information regarding the 
algebraic structure of the fundamental group of the manifold, as well as on the action of $G$ on $M$. This information is valid for simple groups of any real rank, including one.  We formulate below a generalization of Gromov's representation theorem \cite{g1}\cite[Thm. 1.5]{zproc}, formulated originally for unimodular structures (where the maximal measurable projective factor is necessarily trivial). In our present context the structure may not be unimodular, and the maximal projective factor gives rise to a certain parabolic subgroup $Q$ which replaces $G$ in the conclusion of the representation theorem, so that when the stationary measure is invariant the results coincide. 

Before stating this generalization we introduce the following notation. Let $Q$ be any proper parabolic subgroup 
of $G$, and write the Langlands decomposition as $Q=M_Q A_Q N_Q$. Denote by 
$Q_0=M_Q^\prime A_Q N_Q$ the co-compact subgroup of $Q$ obtained by requiring that $M_Q^\prime$ contains exactly  the simple factors of $M_Q$ which are non-compact. 
We recall that any stationary measure $\nu$ has a representation as a convolution $\nu=\tilde{\nu}_0\ast \lambda$, where $\lambda$ is a uniquely determined $P$-invariant probability 
measure, $P$  a minimal 
parabolic subgroup, and $\tilde{\nu}_0$ is a probability measure on $G$. Furthermore, we recall that 
$(M,\nu)$ has a maximal (measurable) projective factor $G/Q(X.\nu)$, where $Q(X,\nu)$ is a uniquely determined standard parabolic subgroup. In general, if we denote by $Q=Q_\lambda$ the stability group 
of the measure $\lambda$, we have $Q_\lambda\subset Q(X,\nu)$ and $Q_\lambda$ may be a proper subgroup. 
The extension $M\to G/Q(X,\nu)$ is called a measure-preserving extension if the measure $\lambda$ is indeed invariant under $Q(X,\nu)$  
(see \cite{nz1} and \cite{nz3} for a full discussion).  
We refer to \S 3 for further details on the notation and terminology used in the statement below.     

Our second main result is as follows : 

\begin{theorem}\label{Q}{\bf Generalization of Gromov's representation theorem.}
Let $(M,\omega,G,\nu)$ be a rigid analytic stationary system, and assume that it is a measure-preserving 
extension of its maximal projective factor. Consider the image $\Gamma$ of the Gromov representation of the fundamental group $\pi_1(M)$  in 
$GL(\mathfrak z)$, where $\mathfrak z$ is the Lie algebra of Killing fields on $\tilde{M}$ commuting with 
$\mathfrak g$.  Let  $\overline{\Gamma}$ denote the Zariski closure of $\Gamma$.  
Then precisely one of the following holds :
\begin{enumerate}
\item The $G$-action is generically locally free, and then  the Lie algebra of $\overline{\Gamma} $ 
has $\mathfrak q_0$ as a sub-quotient. Namely $\text{Lie}(\overline{\Gamma})$ contains a sub-algebra which has a factor algebra  containing an isomorphic copy of  the Lie algebra $\mathfrak q_0$ of the group $Q_0$ associated with 
$Q$. 
\item On a generic set $M_0$, the connected components of all stability groups are mutually conjugate and of positive dimension, the generic smooth maximal projective factor $G/Q_s$ is non-trivial, and $M_0$ is diffeomorphic to the induced action $G\times_{Q_s} Y$, where $Y\subset M_0$ is a $Q_s$-invariant submanifold.      
\end{enumerate}  
 
\end{theorem}

%

Note that when the stationary measure is $G$-invariant, the first case  of Theorem \ref{Q} applies and we retrieve a version of Gromov's representation theorem. 

It may happen that the stability group $Q_\lambda$ reduces to the minimal parabolic subgroup $P$. 
But it is often the case that $Q_\lambda$ is as large as it can possibly be, namely $Q_\lambda=Q(X,\nu)$. This condition is equivalent to $(M,\nu)$ being a measure-preserving extension of its maximal (measurable) projective factor.  We recall that this holds whenever $G$ is of real rank at least two and the action of $G$ on $(M,\nu)$ is irreducible 
in the sense defined by \cite{nz3}. In particular, this holds when the $P$-action on $(M,\lambda)$ is mixing \cite{nz1}. 

We refer to \S 6 for some further results and corollaries 
of Theorem \ref{main} and Theorem \ref{Q}.

As to the methods of proof, let us note the following 
essential ingredients. First, we will show that it is possible to generalize Gromov's centralizer theorem beyond the context of simple groups preserving a unimodular structure. 
We will establish such a generalization below, relying on and generalizing some of the arguments presented in \cite{zproc}. Second, we note that in the presence of an invariant rigid structure of algebraic type, a number of equivariant maps to projective varieties present themselves naturally. Such a map immediately produces an equivariant projective factor, unless it is generically  constant. 
We will formulate this below as a general dichotomy, and apply it to the generic algebraic hull, to the Lie algebra of a generic stability  
group, and to several maps defined by evaluation of Killing vector fields. We will thus obtain certain useful reductions which will be utilized in the course of the proof of Theorems  \ref{main} and \ref{Q}. 

\vskip0.1truein
\noindent {\bf Acknowledgements.} The authors would like to thank David Fisher, Charles Frances 
and Karin Melnick for their very useful comments. 
\section{The generic smooth algebraic hull and maximal  projective factor}

\subsection{2.1 The generic smooth algebraic hull}

Let $M$ be a manifold, $G$ a Lie group acting smoothly on $M$, and
let $\eta$ be a $G$-quasi-invariant finite Borel measure on $M$ of full support.

A subset $M_0\subset M$ is called generic if it is open, dense and co-null.  
A map $f: M\to N$ is called generic if it defined on some generic 
subset $M_0\subset M$ and is smooth on $M_0$. If $M$ and $N$ are $G$-spaces 
a $G$-equivariant generic map is a generic map defined on some
 $G$-invariant generic subset satisfying $f(gm_0)=gf(m_0)$, for 
$g\in G$ and $m_0\in M_0$.

If $P \rightarrow M$ is a principal $H$-bundle, then by a
 {\em generic reduction} of $P$ to a subgroup $L \subseteq H$ we mean a 
smooth reduction defined on a generic set, i.e., 
a smooth $H$-equivariant map $P_1 \rightarrow H/L$, 
where $P_1$ is an $H$-invariant generic set.  
If $G$ acts on $P$ by principal bundle automorphisms, 
a $G$-invariant generic reduction is a $G$-invariant 
reduction  defined on a $G$-invariant generic set $P_1$.  

\begin{prop}
\label{algebraic hull}{\bf Generic smooth algebraic hull.}
Assume $G$ acts topologically transitively on $M$ (for example, ergodically on $(M,\eta)$, where $\eta$ has full support), 
and $P\to M$ is a principal 
$H$-bundle on which $G$ acts by bundle automorphisms.
Suppose $H$ is real algebraic.  
Then there is an algebraic subgroup $L \subseteq H$ such that
\begin{enumerate}
\item\label{Pn2.1(i)} There is a $G$-invariant generic reduction of $P$ to $L$.
\item\label{Pn2.1(ii)} If $L' \subsetneq L$ and $L'$ is algebraic, then there is 
no such reduction.
\item $L$ is unique up to conjugacy in $H$ \wrt properties
 \pref{Pn2.1(i)} and \pref{Pn2.1(ii)}.
\item If there is a $G$-invariant generic reduction of $P$ 
to $J$, then some conjugate of $J$ contains $L$.
\end{enumerate}
\end{prop}

\begin{definition}
\label{hull-def}
The group $L$ is called the generic smooth (or $C^\infty$) 
algebraic hull of the $G$-action on $P$.
\end{definition}

\begin{remark}
\begin{enumerate}
\item 
Proposition \ref{algebraic hull}, its proof and Definition \ref{hull-def}
 are analogues of
 the measurable algebraic hull (see e.g. \cite{zbook}).  
The generic smooth algebraic hull was first introduced in \cite{topsr}.
\item Similar arguments establish the existence of the $C^r$ generic algebraic 
hull, of an action on a principal bundle, for any $r \ge 0$. A full discussion can be found in \cite{feres}.  
\end{enumerate}
\end{remark}

Let us illustrate this notion by the following example, which will appear repeatedly below. 
Let $G$ be a Lie group, and consider its action on 
 $M \times \mathfrak g$, the product vector bundle, where 
$G$ acts on $\mathfrak g$ via $\Ad$.
Consider the associated principal bundle $M\times GL(\mathfrak g)$. 
Then the generic smooth algebraic hull 
 $L \subseteq \GL(\mathfrak g)$ obviously 
satisfies $L \subseteq \text{Zcl}(\Ad(G))$, 
Zcl denoting the closure in the Zariski topology.

 Recall also that the image of a real algebraic group under a rational representation is of finite index in its Zariski closure. We will apply this fact for  $G$ semisimple or for the group $A_QN_Q$ (where $Q=M_QA_QN_Q $ is a parabolic subgroup), and the representations   $\Ad(G)$ and $\Ad_{\mathfrak g}(A_QN_Q)$.

\subsection{2.2 Generic smooth maximal projective factor}

Recall that a compactly supported 
probability measure $\mu$ on an lcsc group $G$ is called admissible if 
some convolution power is absolutely continuous with respect to Haar measure, and 
the support of $\mu$ generates $G$. If $G$ acts measurably on a standard Borel space 
$X$, a Borel probability measure $\nu$ is called $\mu$-stationary if $\mu\ast\nu=\nu$, namely 
for every $f\in L^1(X,\nu)$
$$\int_G\int_X f(gx)d\nu(x)d\mu(g)=\int_X f(x)d\nu(x)\,\,.$$ 
The pair $(X,\nu)$ is called a $(G,\mu)$-space. We can now state 

\begin{prop}\label{max factor}{\bf Generic smooth maximal projective factor}
Let $G$ be a connected non-compact semisimple Lie group with finite center, and $M$ a manifold on which $G$ acts smoothly.
Let $\mu$ be an admissible 
probability measure on $G$, and $\nu$ an ergodic $\mu$-stationary measure of full support. Then there exists 
a parabolic subgroup $Q\subset G$, with the following properties :
\begin{enumerate}
\item There exists a generic smooth $G$-equivariant factor map $\phi : M_0\to G/Q$, satisfying also $\phi_\ast(\nu)=\nu_0$, 
where $M_0\subset M$ is generic and $G$-invariant,  $\nu_0$  the unique $\mu$-stationary measure on $G/Q$.
\item $G/Q$ is the unique generic smooth maximal homogeneous projective factor of $M$, namely any other such factor
  of $M$ is also a factor of $G/Q$. In particular the parabolic subgroup $Q$ is uniquely determined up to conjugation. 
  \item The map $\phi$ is generically unique, namely any other $G$-equivariant generic smooth map $\psi : M_1\to G/Q$ 
  agrees with $\phi$ on a $G$-invariant generic set.   
\end{enumerate} 
\end{prop}
\begin{proof}
The existence of a unique maximal homogeneous projective factor was proved for the measurable category 
in \cite[Lemma 0.1]{nz3}, and the same proof applies without change in the present set-up as well. 
\end{proof}
\begin{definition}
Under the assumptions of Proposition \ref{max factor}, $G/Q$ is called the generic smooth maximal projective factor of 
the $G$-action on $(M,\nu)$. 
\end{definition}

\subsection{2.3 The smooth algebraic hull along orbits}

In our analysis below, $G$ will act smoothly on a manifold $M$ with tangent bundle $TM$, and 
we will have cause to be interested in the sub-bundle $TO$ consisting of the tangents to the $G$-orbits, 
particularly when the $G$-action is locally free. An important role will be played by the generic 
smooth algebraic hull along the orbits, particularly when the action is generically locally free. 
Recall first that for a Lie group $G$ and the associated principal bundle $M\times GL(\mathfrak g)$,  
(a conjugate of) the generic smooth algebraic hull  $L$ must be contained in 
 $\text{Zcl}(\Ad(G)) \subseteq \GL(\mathfrak g)$. 

We can now state the following result, which  will play an important role below.

\begin{prop}\label{hull-Ad}{\bf Generic smooth algebraic hull along orbits}.

 Let $G$ be a connected semisimple Lie group with finite center and no compact factors, acting smoothly 
 on a manifold $M$. Let $\mu$ be an admissible probability measure on $G$. 
\begin{enumerate}
\item If $\nu$ is an ergodic $\mu$-stationary measure on $M$ of full support,  then  
 the generic smooth algebraic hull $L$ for the action of $G$ on the bundle $M\times GL(\mathfrak g)$ 
is of finite index in $ \text{Zcl}(\Ad(G))$ if and only if the generic smooth maximal projective factor of $(M,\nu)$ is trivial.
\item Let $Q\subset G$ be a proper parabolic subgroup with Langlands decomposition $Q=M_QA_QN_Q$.  The measurable algebraic hull $L_0$ of the action of the group  
$A_QN_Q\subset Q$ on $M\times GL(\mathfrak g)$  w.r.t. an ergodic invariant probability 
measure $\lambda$ on $M$ is of finite index in $\text{Zcl}(\Ad_{\mathfrak g}(A_QN_Q))$.
\end{enumerate}
\end{prop}
\begin{proof}
\begin{enumerate}\item If $L$ is of finite index in $\text{Zcl}(\Ad(G))$, then $M$ does not admit a generic smooth equivariant map 
to any variety of the form $G/H$ if $H$ is algebraic of positive codimension. In particular, there exists such a map to the variety $G/Q$ only if $Q=G$. Conversely, if $L$ has positive codimension in $\text{Zcl}(\Ad(G))$ then $M$ does admit a non-trivial projective factor. Indeed, it follows from Proposition 3.2 of \cite{nz3}  that any positive dimensional algebraic 
variety of the form $G/L$ which supports a $\mu$-stationary measure must factor over $G/Q$ for some 
$Q\subsetneq G$. 
\item  $L_0\subset \text{Zcl}(\Ad_{\mathfrak g}(A_QN_Q))$ and there exists an equivariant map 
$M\to \text{Zcl}(\Ad_{\mathfrak g}(A_QN_Q))/L_0$, so 
the algebraic variety $\text{Zcl}(\Ad_{\mathfrak g}(A_QN_Q))/L_0$ has an 
$A_QN_Q$-invariant probability measure, 
the image of $\lambda$. As is well-known,  $A_QN_Q$ does not have algebraic subgroups of finite covolume, and any $A_QN_Q$-invariant probability measure on any algebraic variety must be supported on fixed points of $A_QN_Q$ in its action on the variety. On the other hand $A_QN_Q$ has finitely many orbits on the variety, since it is of finite index in its Zariski closure. It follows that $L_0$ contains $\Ad_{\mathfrak g}(A_Q N_Q)$ and is of finite index.  
\end{enumerate}
\end{proof}

\subsection{2.4 Projective factors and local freeness}
Let $G$ be a connected lie group acting smoothly on a manifold $M$. For any $m\in M$, we can consider the stability group $st_G(m)$, and its Lie algebra $\mathfrak s\mathfrak t _G(m)$.
Consider the map $m\mapsto \mathfrak s\mathfrak t _G(m)$,
 from $M$ to $Gr(\mathfrak g)=\coprod_{k=0}^{\dim G}  Gr_k(\mathfrak g)$, the disjoint union of the Grassmann varieties of $k$-dimensional subspaces 
 of $\mathfrak g$. This map is obviously $G$-equivariant (where $G$ acts on $\mathfrak g$ via $\Ad$), and measurable (see e.g. \cite[\S 3]{nz3} or \cite{dani}). Given an ergodic $G$-quasi-invariant measure $\eta$ on $M$, the values of this map 
 will be almost surely concentrated in a single $G$-orbit  in $Gr(\mathfrak g)$, since the action of $G$ on 
 $Gr(\mathfrak g)$ which is a disjoint union of algebraic varieties is (measure-theoretically) tame. 
 
 Now note that if the measure has full support, we can deduce that on a $G$-invariant {\it generic} set $M_0$ the Lie algebras of stability groups of points in $M_0$ are all in the same $\Ad(G)$-orbit, and in particular have the same dimension, say $d_0$. 
 Indeed, the $G$-invariant sets $U_d=\{m\in M\,;\, \dim \mathfrak s \mathfrak t _G(m)\le d\}$ are open sets, since the dimension function is upper semi continuous (see \cite[\S 2]{stuck}). Now  
 $U_{d_0-1}$ is an open and hence must be empty, 
 since the measure assigns positive measure to every non-empty open set, but the dimension of the stability group is almost surely $d_0$. Thus  
 $U_{d_0}=\{m\,;\, \dim\mathfrak s \mathfrak t_G(m)=d_0\}$ is a conull, $G$-invariant open and 
 dense set. 
 
Denoting the latter set by $M_0$,  the map $m\mapsto \mathfrak s\mathfrak t _G(m)$ is 
smooth on  $M_0$ (see e.g. 
 the argument of \cite[Lemma 2.10]{stuck}). 
 Thus we obtain a generic smooth map $M_0\to G\cdot \mathfrak s$, where $\mathfrak s$ is the Lie algebra 
 of a closed subgroup $S$. We can write $G\cdot \mathfrak s=G/N_G(S^0)$, where $S^0$ is the connected component and $N_G(S^0)$ is its normalizer in $G$. Thus for a connected semisimple Lie group we can conclude that
 $(M,\eta)$ has a non-trivial homogeneous smooth projective factor whenever  
 $N_G(S^0)$ is contained in a non-trivial parabolic subgroup $Q$. 
We can thus conclude the following, which is a generic smooth version of \cite[Prop. 3.1]{nz3}.

 \begin{prop}\label{3 alternatives}{\bf Smooth projective factors and local freeness.}
 
  Let $G$ be a connected semisimple  Lie group with finite center and no compact factors, acting smoothly on a manifold $M$ with ergodic quasi-invariant measure of full support. Then exactly one of the following three alternatives occur :
 \begin{enumerate}
 \item On a $G$-invariant generic set $M_0\subset M$, the identity components of the stability groups $st_G(m)$ are all  
  conjugate (say to $S^0$) and have positive dimension.  The generic smooth maximal projective factor $G/Q$ is non-trivial, and  $N_G(S^0)\subset Q\subsetneq G$.
  \item On a $G$-invariant generic set the identity components of the stability groups are all equal to a fixed normal subgroup $G_1$ of $G$ of positive dimension, so that the action of $G$ on the generic set factors 
  to an action of the group $G/G_1$.   
  \item On a $G$-invariant generic set, the $G$-action is locally 
  free, namely the stability groups are 
  discrete. 
 \end{enumerate}
 \end{prop}  
  
\begin{remark}\label{smooth}  We note that we will have below several other (a priori, measurable) $G$-equivariant maps which assign to points in $M$ linear subspaces of a finite dimensional space, and that the argument preceding Proposition \ref{3 alternatives} is completely general, and shows the following. Given such a map $\phi$, suppose that the dimension function $\dim \phi(m)$ is upper semi continuous, and the function $\phi$ is smooth on open sets where $\dim \phi(m)$ is constant.  If the linear spaces $\phi(m)$ almost surely lie in one $G$-orbit (where the measure on $M$ has full support), then there is a {\it generic} set where $\dim \phi(m)$ is constant and $\phi$ is smooth. This will be applied to the map 
$\psi$ below and others, and the verifications of the required properties is routine. 
 \end{remark}
\section{Rigid geometric structures :  orbits of the centralizer}

Let $G$ be a connected Lie group acting
 on a connected manifold $M$, preserving 
a rigid analytic structure of algebraic type \cite{g1}, which we denote by $\omega$.
Of course, we do not assume $\omega$ is unimodular. 
 We let $\eta$ denote a $G$-quasi-invariant finite Borel measure 
on $M$, usually assumed to be of full support.

We will follow the notation and terminology 
of \cite{zproc}, to which we refer for further discussion and the results cited below. 
Let $\tilde{\omega}$ be the structure on $\tilde{M}$ 
lifted from $\omega$, and $\mathfrak k$ the space of global vector fields on 
$\tilde{M}$ preserving $\tilde{\omega}$, namely global Killing fields. Let $\mathfrak n$ be the 
normalizer of $\mathfrak g$ in $\mathfrak k$, 
 where we have $\mathfrak g \subseteq \mathfrak k$ by lifting the action 
of $G$ to its universal covering group $\tilde{G}$. Let 
$\mathfrak z$ be the centralizer of $\mathfrak g$ in $\mathfrak k$.  
Thus $\mathfrak z \subseteq \mathfrak n \subseteq \mathfrak k$.  
Since $\mathfrak g$ is semi-simple, 
$\mathfrak n = \mathfrak g \oplus \mathfrak z$.

$\pi_1(M)$ acts on $\mathfrak k$, and $\mathfrak n$ and 
$\mathfrak z$ are $\pi_1(M)$-submodules.  
Elements of $\mathfrak k$ are lifted from a vector field on $M$ 
if and only if they are $\pi_1(M)$-invariant.  
$\tilde{G}$ also acts on $\mathfrak k$, the derived $\mathfrak g$-module 
structure being simply $\ad_{\mathfrak k}\bigr|_{\mathfrak g}$.  
By definition, $\tilde{G}$ commutes with $\mathfrak z$.  
The $\pi_1(M)$- and $\tilde{G}$-actions commute on $\tilde{M}$ and, 
in particular, on $\mathfrak k$ and its submodules.
 If $G$ acts analytically on the analytic manifold $M$, 
rigidity of the structure $\omega$ implies that the space of analytic 
Killing vector fields $\mathfrak k$ 
is finite dimensional and in fact $\text{Aut}(M,\omega)$ is a Lie group acting analytically on $M$,  
see \cite{g1}, \cite[Thm. 2.3]{zproc}.

For each $y \in \tilde{M}$, let 
$\Ev_y : \mathfrak k \rightarrow T\tilde{M}_y$ be the evaluation map.  
Write $$\psi(y) = \{X \in \mathfrak g : X(y) \in \Ev_y(\mathfrak z)\}
 \subseteq \mathfrak g.$$  Then $\psi$ is $\pi_1(M)$-invariant and 
$G$-equivariant (where $G$ acts on $\mathfrak g$ by $\Ad$), 
and hence factors through a generic $G$-map 
$M \rightarrow \Gr(\mathfrak g)$.
If the structure $\omega$ is unimodular and $M$ compact, so that there exists a 
$G$-invariant finite volume form, then Gromov's centralizer theorem 
is the assertion that, on a generic set, $\psi(y) = \mathfrak g$, provided $G$ is simple.  
This is a key property in the analysis of \cite{g1} and \cite{zproc}.

We will  generalize this in three  respects, allowing non-unimodular structures, stationary measures rather than invariant ones, and groups more general than simple ones when a 
probability measure is assumed to be preserved. We begin with 
formulating the following generalization of Gromov's centralizer theorem
\cite{g1}\cite{zproc}(Cor. 4.3). Consider the rigid structure of algebraic type 
$\omega'$ on $M$ consisting of the pair 
$(\omega, \mathfrak g)$.  A local automorphism of 
this structure is a local diffeomorphism which preserves 
$\omega$ and normalizes $\mathfrak g$.  We define global 
and infinitesimal (to order $k$) such automorphisms 
similarly.  Let $M_0 \subseteq M$ be the set of points for which 
infinitesimal automorphisms of $\omega'$ of sufficiently high fixed 
order at $m \in M_0$ extend to local automorphisms of $\omega'$.  
[cf. the discussion in \cite{g1} and \cite{zproc}(Theorem 2.1).]  
Then $M_0$ contains a generic set, and in addition is clearly 
$G$-invariant, so we can in fact 
assume $M_0$ is generic and $G$-invariant.

\begin{theorem}\label{Evl=Evh}{\bf Infinitesimal orbits of the centralizer for general groups.}
 Let $(M,\omega)$ be a connected analytic  compact  manifold  with a rigid analytic structure of algebraic type. 
Let $H\subset Aut(M,\omega)$ be a connected real algebraic group 
with finite center,  and no proper real algebraic subgroups of 
finite co-volume.  Assume $H$ preserves 
an ergodic probability measure $\lambda$ and is essentially locally free.
 Let $\mathfrak z$ be the centralizer of $\mathfrak h$ in the normalizer subalgebra 
 $\mathfrak n$.  Then, for $\lambda$-\ale $y \in M$, 
$\Ev_y(\mathfrak z) \supseteq \Ev_y(\mathfrak h)$.

\end{theorem}
\begin{proof}
Gromov's argument 
to prove the centralizer theorem in the case of 
invariant measure does not depend on semi-simplicity of $G$. 
 Rather, it uses two crucial properties of the group, namely that the adjoint representation has finite kernel  and that the Borel 
density theorem hold (i.e., that the group should have no algebraic subgroups of co-finite volume.) 
 In addition, the argument requires essential local freeness of the action, which holds for any 
 {\it measure-preserving} action of a simple group.  
 Our assumptions about the group $H$ and its action on $M$ allow the same argument to proceed without change. We refer to  \cite[Thm 4.2, Cor. 4.3]{zproc} or \cite{g1} for details. 
 \end{proof}

Theorem \ref{Evl=Evh} will be applied to the action of a parabolic subgroup $Q$ of a simple group $G$, preserving a probability measure on the manifold. In this case, we obtain the following simple relation 
between global Killing fields centralizing the $Q$-action and those centralizing the $G$-action. 

\begin{lemma}\label{epimorphic}
\label{l=z}
Let $Q=M_Q A_Q N_Q$ be a proper parabolic subgroup of a connected semisimple Lie group $G$ 
with finite center and no compact factors. Let $G$ act on a 
compact manifold preserving a rigid analytic structure of algebraic type. 
If a Killing field on $\tilde{M}$ normalizes $\mathfrak g$ and 
centralizes ${\mathfrak a}_Q\oplus {\mathfrak n}_Q$, 
then it also centralizes $\mathfrak g$, 
namely $\mathfrak l = \mathfrak z $.
\end{lemma}

\begin{proof}
The space $\mathfrak k$ of global Killing fields is a finite-dimensional $\mathfrak g$-module, 
so any vector invariant under $ {\mathfrak a}_Q\oplus {\mathfrak n}_Q$ is 
also invariant under $\mathfrak g$. This follows from the fact that for a proper parabolic subgroup $Q$ 
the subgroup $A_Q N_Q$ is epimorphic in $G$, namely has the same set of invariant vectors as $G$ does, in any finite dimensional representation.  
\end{proof}

Consider now a semisimple Lie group $G \subset \text{Aut} (M,\omega)$, where now the underlying measure $\nu$ on $M$ is assumed to be stationary, for some admissible measure $\mu$ on $G$, namely consider rigid analytic stationary systems.  
Note that the $G$-equivariant map $\psi : M\to Gr(\mathfrak g)$ defined above (via evaluation) always exists, and allows us to modify Gromov's techniques from \cite{g1} to prove the following.

\begin{theorem}\label{2 alternatives}{\bf Projective factors and the evaluation map.}
\label{psi-neq-0}

 Let $(M,\omega, G,\nu)$ be a rigid analytic stationary system.  Then at least one of the following 
possibilities holds : 
\begin{enumerate}
\item for $y$ in a $G$-invariant generic set, $\psi(y)=\mathfrak g$. 
\item  $M$ has a generic smooth maximal projective factor which is non-trivial. 
\end{enumerate} 
\end{theorem}
\begin{proof}
We apply the trichotomy provided by Proposition \ref{3 alternatives} and the results above, as follows. 

First, since $\nu$ is stationary, we can let $\lambda$ be the corresponding 
$P$-invariant measure, where $P$ is a minimal parabolic subgroup (see \cite{nz1} or \cite{nz4} for a full discussion).  Then $\nu =\tilde{\nu}_0 \ast\lambda$, 
where $\tilde{\nu}_0$ is a (compactly supported, absolutely continuous)  
probability measure on $G$.  
 
 Consider first the case where  $G$ is essentially locally free on 
$(M, \nu)$ (and thus also on $\tilde{M}$). It then follows that, for $\tilde{\nu}_0$-\ale $g \in G$, 
$gPg^{-1}$ is essentially locally free on $(M, g_*\lambda)$.  Furthermore, 
for $\tilde{\nu}_0$-almost every $g$, 
 $g_*\lambda(M_0) = 1$, if $\nu(M_0) = 1$ for the generic set $M_0$.  
Thus, replacing $P$ and $\lambda$ by conjugates under some $g \in G$ 
if necessary, we may assume that $P$ is essentially locally free and 
$\lambda(M_0) = 1$. Applying Theorem \ref{Evl=Evh} to the action of $A_PN_P$, together with 
Lemma \ref{epimorphic} we conclude that $\psi(y)\neq 0$ for $\lambda$-almost all points $y\in M_0$.
Then since $\tilde{\nu}_0 \ast \lambda = \nu$, 
it follows easily that $\psi(y) \neq 0$ for $\nu$-almost every $y \in M_0$.
Now clearly $\dim \psi(m)$ is $\nu$-almost surely constant and if this constant is smaller than 
$\dim \frak g$
then we can apply the arguments noted in Remark \ref{smooth} to $\psi$ and conclude that 
we have a non-trivial generic smooth maximal projective factor.

Otherwise $\psi(y)=\mathfrak g$ 
$\nu$-almost surely, and we can again apply the argument of Remark \ref{smooth} 
to conclude that on a generic set $\psi(y)= \mathfrak g$, so that case (1) obtains. 

We can now therefore consider the case where $G$ is not essentially locally free. Since $\nu$ is stationary, any generic smooth quotient $G/L$ with $L$ 
a proper algebraic subgroup has a non-trivial homogeneous 
projective factor, i.e. a factor of the form $G/Q$ with $Q\subsetneq G$ 
a parabolic subgroup, as we saw in Proposition \ref{3 alternatives}. We therefore obtain a generic smooth non-trivial maximal projective factor, unless the stability groups are normal on a generic set. 
This is ruled out by our assumption that the $G$-action is faithful and the proof is complete. 

\end{proof}

Thus, Theorem \ref{psi-neq-0} allows us to reduce to the 
case $\psi(y) = \mathfrak g$ generically, a fact we will use below.

\section{Stability subalgebras of the normalizer}

We now add some further preparations before taking up the proof of Theorem \ref{main}. We return to considering $TM$, and from now on assume that the $G$-action is generically  
locally free, otherwise there already exists a non-trivial generic smooth maximal projective factor, 
according to Proposition \ref{3 alternatives} and our assumptions in Theorem \ref{main}. 
Then on a generic set we can identify the tangent bundle to the 
orbits of $G$, say $T{\cal O} \subseteq TM$ with 
$M \times \mathfrak g$, where $G$ acts on $\mathfrak g$ via $\Ad$.  
By Proposition  \ref{hull-Ad}, either $M$ has a non-trivial generic smooth projective factor  or 
$\text{Zcl}(\Ad(G))$ coincides (up to finite index) with the generic smooth algebraic hull of the action on 
$M \times \mathfrak g$, and hence on $T{\cal O}$.  
Similarly, consider the extension of the $G$-action to an action  on the principal bundle $P^{(k)}(M)$ whose fiber at each point is the space of $k$-frames.  Again, the algebraic hull of this action contains a group locally isomorphic to $G$.  This therefore enables us to obtain the 
conclusions of \cite[Thm. 4.2]{zproc} 
exactly as in the finite 
measure-preserving case.  Namely, we can assume that the smooth generic algebraic hull on $M\times \mathfrak g$ (and $T{\cal O}$) is in fact the Zariski closure of $\Ad(G)$ in 
$GL(\mathfrak g)$ (otherwise the generic smooth maximal projective factor is already non-trivial).

 \begin{remark} As noted already in the introduction,  we are repeatedly 
using the technique of asserting that either there is a non trivial projective factor, 
or we have a conclusion similar to the measure-preserving case. This was already applied to three maps from $M$ to projective varieties, namely 
to the map associating to a point the Lie algebra of its stability group, to the map $M\to G/L$ 
where $L$ is the algebraic hull on the tangent bundle to the $G$-orbits, and to the
 map $\psi$ to the Grassmann variety $Gr(\mathfrak g)$. We will also use this technique once more in the proof of Theorem \ref{main} below, applying it to the map  $y\mapsto {\mathfrak n}_y$
  defined by the kernel of the evaluation map, where $y\in \tilde{M}$. 
Of course, in the measure-preserving case, Borel density is used to
 eliminate the possibility of a projective quotient from the outset.
 \end{remark}
 
This technique yield also the following results, which we will use below. 

\begin{prop}
\label{stability algebras}{\bf Stability Lie algebras.}

Let $(M,\omega, G,\nu)$ be a rigid analytic stationary system. 
Then either the generic smooth maximal projective factor is non-trivial, or  
for $y_0$ in a generic set, there exists a Lie algebra 
$\mathfrak g_\Delta \subseteq \mathfrak n$ with the following properties:
\begin{enumerate}
\item All elements of $\mathfrak g_\Delta$ vanish at $y_0$; i.e., 
$\mathfrak g_\Delta \subseteq \mathfrak n_{0}$ (${\mathfrak n}_0={\mathfrak n}_{y_0}$ is 
the kernel of the evaluation map $X\mapsto X_{y_0}$).
\item $\mathfrak g_\Delta$ is isomorphic to $\mathfrak g$, 
and in fact is the graph in $\mathfrak n = \mathfrak g \oplus \mathfrak z$ 
of an injective homomorphism $\sigma : \mathfrak g \rightarrow \mathfrak z$. 
 We let $\mathfrak g' = \sigma(\mathfrak g)$ be the image in $\mathfrak z$.
\item Let ${\mathfrak z}_0={\mathfrak z}_{y_0}$ denote the kernel of the evaluation map.
Let $L$ be the generic smooth algebraic hull of the $G$-action on the principal $GL(\mathfrak n)$-bundle 
over $M$ associated with the representation of $\pi_1(M)\to \Gamma\subset GL(\mathfrak n)$. 
 Then  $\text{ad}_{\mathfrak n}(\mathfrak g_\Delta)$ acting on 
$\mathfrak z/\mathfrak z_0$ is contained in the Lie algebra 
$\Lie(L)$ of the algebraic hull, and $L$ is contained in $ \overline{\Gamma}=\text{Zcl}(\pi_1(M))\subset GL(\mathfrak n)$. 
 In particular, 
$\ad_{\mathfrak z}(\mathfrak g_\Delta) = \mathfrak g' \subseteq \Lie(L)
 \subseteq \Der(\mathfrak z; \mathfrak z_0)$, 
where the latter is the set of derivations of $\mathfrak z$ leaving $\mathfrak z_0$ invariant.
\end{enumerate}
\end{prop}

\begin{proof}
Starting with a 
simple $G$, by Proposition \ref{hull-Ad} if the generic smooth maximal projective factor is trivial, then the action is generically locally free, and the generic 
smooth algebraic hull of the tangent bundle to the orbits coincides with $\text{Zcl}(\Ad(G))$. 
Given these facts, the method of proof of \cite[Thm. 4.2, Cor. 4.3, Thm. 1.5]{zproc}  applies without change.
  \end{proof}
  
  We let $G' \subseteq L$ and $G_\Delta \subset  G \times L$ be the 
corresponding groups, where $G_\Delta$ is the graph of $\sigma : G \rightarrow G' \subseteq L$.

\begin{remark}
Let us note that a conclusion similar to that of Proposition \ref{stability algebras} holds for any connected Lie group $H$ with finite center, provided its action is ergodic w.r.t. a quasi-invariant measure of full support, generically locally free, and the generic smooth algebraic hull along the orbits coincides with with $\text{Zcl}(\Ad(H))$ (up to finite index). Under these conditions either the rigid analytic system admits an $H$-equivariant  non-constant map to a projective variety, or the conclusions 1, 2 and 3 holds. 
When $H$ is simple,  such a map necessarily gives rise to a {\it compact homogeneous} projective factor, as stated.   
\end{remark}

We now note that the following result on properness of the action on the universal cover, in analogy with   \cite[Cor. 4.5]{zproc}, which proves properness 
in the unimodular,  measure-preserving case.

\begin{prop}\label{properness}{\bf Generic properness on the universal cover}.

Let $(M,\omega, G,\nu)$ be a rigid analytic stationary system. 
Then either the generic smooth maximal projective factor is non-trivial, 
or the $\tilde{G}$-action on $\tilde{M}$ is (locally free and) proper on a generic set $M_0$. Namely, generically the stability groups in $\tilde{G}$ of points in $\tilde{M}$ are compact, 
and the orbit map $\tilde{G}/\tilde{G}_y\to \tilde{G}\cdot y$ is a diffeomorphism. The same conclusion obtains also 
under the assumptions in the last part  of Proposition 
\ref{stability algebras}. 
\end{prop}
\begin{proof}
Assume the generic smooth maximal projective factor is trivial.  We first claim that  the stability groups are generically compact (so in our situation, finite), and the orbit maps  $\tilde{G}/\tilde{G}_y \mapsto \tilde{G}\cdot y \subset \tilde{M}$, are homeomorphisms (so that orbits are locally closed). Indeed, this follows from the argument in \cite[Cor. 4.5]{zproc}, which uses only generic local freeness of the $G$-action, together with the fact that generically $\psi(y)=\mathfrak g$, and its consequence that the local orbits of the centralizer covers the local orbits of $\tilde{G}$.  Now if the orbit $\tilde{G}y$ is locally closed, then it is a regular submanifold of $\tilde{M}$ 
 (see \cite[Thm. 2.9.7]{Va}), of dimension $\dim G$ and the orbit map is a diffeomorphism.  
\end{proof}

\section{Proof of Theorem \ref{main}, Theorem \ref{rank1} and Theorem \ref{Q}}

{\bf Proof of Theorem \ref{main}}. 

Theorem \ref{main} follows immediately from Proposition \ref{stability algebras}. Indeed,  assume that the generic smooth maximal projective factor of the rigid analytic stationary system is trivial. Then the group $L$ defined as the generic smooth algebraic hull of the $G$-action on the principal $GL({\mathfrak n})$-bundle associated with the representation  $\pi_1(M)\to \Gamma\subset GL({\mathfrak n})$ has the property that it is contained in the Zariski closure $\overline{\Gamma}$ in $GL({\mathfrak n})$. On the other hand, the Lie algebra of $L$ contains the Lie algebra ${\mathfrak g}^\prime$ which is an image of $\mathfrak g$ under an injective homomorphism, and hence is isomorphic to $\mathfrak g$.    

The last statement follows from the fact that the Zariski closure of a virtually solvable group has a solvable Lie algebra, which therefore cannot contain a copy of the simple Lie algebra of  $\mathfrak g$ as a subalgebra. \qed

\vskip0.1truein

\noindent {\bf Proof of Theorem \ref{rank1}}.

 Fix $G=SO^0(n,1)$, $n\ge 2$, and let $\Gamma$ be a uniform lattice 
in $G$ which maps onto a free group on $k\ge 2$ generators. Note that any uniform lattice in $G$ admits a finite index subgroup which maps onto a non-Abelian free group. Now fix a uniform lattice $\Delta\subset SO^0(1,m)$, $m\ge 4$ which has a generating set with $k$ elements. Consider the surjective homomorphism from $\Gamma$ to $\Delta$ obtained by composition, and let $\Gamma $ act on the $(m-1)$-dimensional sphere via the action of $\Delta \subset SO^0(1,m)$ as a group of conformal transformations.  As is well-known, the $\Delta$-action, and hence also the $\Gamma$ action, is minimal and has no invariant probability measure.  Let $M$ denote the action of $G$ induced by the action of $\Gamma$ on the $(m-1)$-dimensional sphere. It has been shown in \cite[Thm. B]{nz1} that $G$ has a unique stationary measure on $M$, which is thus ergodic. This measure is not invariant, since the sphere does not admit a $\Gamma$-invariant measure. Since the $\Gamma$-action is minimal on the sphere, the $G$-action is minimal on $M$, and hence any $G$-quasi-invariant  measure has full support. $M$ has the structure of a smooth bundle over $G/\Gamma$ with typical fiber an $(m-1)$-dimensional sphere, and $M$ also has the structure of an analytic manifold in the obvious way. To produce an invariant rigid analytic geometric structure, take the product of the pseudo-Riemannain structure associated with the Killing form on $G/\Gamma$, and the conformal structure on the sphere. When $m\ge 4$ the conformal structure is of finite type and hence rigid, and we refer to \cite[\S 3]{fz} for a full discussion of the rigidity of the product structure. This concludes the proof of Theorem \ref{rank1}. 
\qed

\vskip0.1truein
\noindent {\bf Proof of Theorem \ref{Q}}. 

We divide the proof into the following steps. 
\begin{enumerate}
\item Let the Lie group $G$ act by bundle automorphisms on the principal bundle $\tilde{M}\to M$ with fiber $\pi_1(M)$, with an ergodic stationary measure  on the base $M$ of full support.  Let  $\Gamma$ be the image of the fundamental group $\pi_1(M)$ under the representation into 
$ GL(\mathfrak z)$. Consider the associated principal $GL(\mathfrak z)$-bundle over $ M$ with fiber $GL(\mathfrak z)$, on which $G$ acts by bundle automorphisms. Then the (measurable) algebraic hull $L_\nu$ of the $G$-action on the $GL(\mathfrak z)$-bundle is contained in the Zariski closure $\overline{\Gamma}$ of $\Gamma$ in $GL(\mathfrak z)$ (see  \cite[Prop. 3.4]{zproc}).  To prove Theorem \ref{Q}, we first assume that the action is generically locally free, and then  it suffices to prove that the Lie algebra  ${\mathfrak l}_\nu$ of the algebraic hull $L_\nu$ has a factor algebra which contains (an isomorphic copy of) the Lie algebra $\mathfrak q_0$ of the group $Q_0\subset Q$. We are assuming that 
 $Q=Q_\lambda=Q(X,\nu) $ is the parabolic subgroup stabilizing the measure $\lambda$ 
 canonically associated to $\nu$. Note that when $\nu$ is invariant we have $\nu=\lambda$ and so $Q=G$, and we obtain Gromov's representation theorem (see \cite{g1} or \cite[Thm. 1.5]{zproc}). 

\item  Let us denote by $\Psi$ the sub-bundle of the tangent bundle to $M$ whose fiber at a given point $y\in M$ is $\psi(y)$. Recall that $\psi(y)$ consists of the vectors in the tangent space to the $G$-orbit at $y$ which are obtained via evaluation of the global Killing vector fields in $\mathfrak z$ at $y$. Thus by definition of $\psi$, $\Psi$  is actually a sub-bundle of the tangent bundle to the $G$-orbits in $M$. The (measurable) algebraic hull $L^\prime_\nu$ of the $G$-action on the bundle $\Psi$ is a factor group of $L_\nu$, and we claim that $\mathfrak l^\prime_\nu$ contains $\mathfrak q_0$.
  
To establish that, we consider the ergodic measure-preserving action of $Q_0$ on $(M,\lambda)$. 
$Q_0$ satisfies the condition of Borel density theorem, and hence by Theorem \ref{Evl=Evh}, for $\lambda$-almost all points $y$, the evaluation of the global Killing vectors fields centralizing the action of $Q_0$ covers the local orbits of $Q_0$. But by Theorem \ref{epimorphic}, since $A_QN_Q$ is epimorphic for every proper parabolic subgroup $Q$, it follows that a global Killing field normalizing $\mathfrak g$ and centralizing $\mathfrak q_0$ necessarily centralizes $\mathfrak g$. We conclude that for $\lambda$-almost all points $y$, $\psi(y)$ covers the local orbits of $Q_0$. The $Q_0$-action on $(M,\lambda)$ is  
essentially locally free, so we deduce by the Borel density property that the (measurable) algebraic hull of the $Q_0$-action on $\Psi$ contains the Zariski closure of  $\Ad(Q_0)$, a group locally isomorphic to 
$Q_0$.

\item By assumption the maximal projective factor of $(M,\nu)$ in the measurable category is $G/Q(X,\nu)$, and $M\to G/Q(X,\nu)$ is a measure-preserving extension. Thus the $G$-action on $(M,\nu)$ is induced from the $Q$-action on $(M,\lambda)$, and it follows that the (measurable) algebraic hull of the $G$-action on the bundle $\Psi$ over $(M,\nu)$ contains the algebraic hull of the $Q$-action on the bundle $\Psi$ over $(M,\lambda)$. Thus it also contains the algebraic hull of the $Q_0$-action on the bundle $\Psi$ over $(M,\lambda)$. Taking parts 1) and 2) into account, we have established that $\mathfrak q_0$ 
is contained in the Lie algebra $L^\prime_\nu$, which is a factor of the  subalgebra ${\mathfrak l}_\nu$ of the Lie algebra of  $\overline{\Gamma}$.

\item Finally, if the action is not generically free, then the first case of Proposition \ref{3 alternatives} applies, and generically, the connected component of every stability group is conjugate to a fixed connected group $S$, with $N_G(S)\subset Q_s$, $Q_s$ 
defining the generic smooth maximal projective factor, which is of course non-trivial in this case.  
Let $\phi: M_0\to G/Q_s$ be the equivariant smooth factor map, and let $Y=\phi^{-1}([Q_s])$. Then 
$Y$ is a smooth $Q_s$-invariant regular submanifold, and $M_0$ is equivariantly diffeomorphic 
to the $G$-action on the manifold $(G\times Y)/Q_s$, namely the $G$-action induced from the action of 
$Q_s$ on $Y$. This fact is general and follows simply from the existence of a smooth equivariant 
map onto a homogeneous space.  
\end{enumerate}
This concludes the proof of Theorem \ref{Q}. \qed 

Let us formulate some corollaries of Theorem \ref{main} and Theorem \ref{Q}, 
 explicating the consequences of restrictions imposed on the fundamental group of $M$. 
 We begin with the following result, establishing that certain rigid analytic stationary actions must have 
 stability groups of positive dimension.  

\begin{theorem}\label{amenable factor}{\bf Manifolds with nilpotent fundamental group.}
Let $(M,\omega, G,\nu)$ be a rigid analytic stationary system, and keep the assumptions of  Theorem \ref{Q}. Suppose further that $\pi_1(M)$ is virtually nilpotent.  Then:
\begin{enumerate}
\item The $G$-action on $(M, \nu)$
 is {\em not} generically locally free.
\item The generic smooth maximal projective factor is non-trivial (and in particular, the stationary measure is not invariant). 
\end{enumerate}
\end{theorem}
\begin{proof}
Note that if the action is generically locally free, then by Theorem \ref{Q} the Zariski closure of the image of $\pi_1(M)$ in $GL(\mathfrak z)$ must have $AN$ as a sub-quotient, where $P=MAN$ is a minimal parabolic subgroup. But $AN$ is a solvable connected Lie group which is not nilpotent, and the Zariski closure of a virtually nilpotent group is virtually nilpotent. 
  \end{proof}

We recall that one case where the stationary system $(M,\nu)$ is a measure-preserving extension 
of its maximal measurable projective factor (so that Theorem \ref{amenable factor} applies) is when the action of $G$ is amenable. Such an action is necessarily induced by an action of a minimal parabolic group (see \cite{nz3} for further discussion). 


%
%
Conversely,  any action of a minimal parabolic subgroup $P$ with invariant probability measure induces 
an amenable action of $G$. However the question of whether the induced action carries a $G$-invariant 
rigid analytic structure if the original $P$-action has an invariant rigid analytic structure  
has not been resolved.  The following consequence of Theorem \ref{amenable factor} is relevant in this direction. 


\begin{cor}{\bf Rigid analytic actions of a minimal parabolic subgroup.}
Let $(M_0,\omega)$ be a compact manifold with rigid analytic geometric structure of algebraic type
and virtually nilpotent fundamental group. Assume $P\subset \text{Aut}(M_0,\omega)$, where $P$ is a minimal parabolic subgroup of a simple non-compact Lie group. Let $\lambda$ be a $P$-invariant ergodic measure on $M_0$ of full support. If the induced action carries a $G$-invariant  rigid analytic structure, then generically points in $M_0$ 
have a stability group of positive dimension. 
\end{cor}
\begin{proof} Assume for contradiction that the action of $P$ on $M_0$ is generically locally free. Then so is the action of $G$ on the manifold $M=M_0\times_P G$, namely the action induced to $G$, with the associated stationary measure (see \cite{nz1}). The manifold $M$ still has a virtually nilpotent fundamental group and a rigid analytic geometric structure of algebraic type preserved by $G$, by assumption. The stationary measure is ergodic and of full support, and $M$ is a measure-preserving extension of $G/P$. Hence Theorem \ref{amenable factor} applies, but the first alternative is ruled out by the fact that $\pi_1(M)$ is virtually nilpotent, and the second alternative is ruled 
out also since the $G$-action is locally free. This contradiction implies that the $P$ action on $M_0$ has stability groups of positive dimension, generically. 
\end{proof}

\begin{remark}

Finally, let us note some consequences that hold whenever the maximal generic smooth projective factor of a rigid analytic stationary system is non-trivial. 
\begin{enumerate}
\item The rigid geometric structure in question cannot be unimodular. Indeed, otherwise it defines an invariant volume form of finite total mass, and this would imply that an invariant probability measure exists on the projective factor, which is not the case. 
\item As a result, under the first alternative in Theorem \ref{main},  or under Theorem \ref{amenable factor}, the rigid structure in question cannot be unimodular.  More generally, $M$ does not carry any invariant probability measure of full support.  
\item It follows also that the top exterior power of the derivative cocycle cannot be cohomologous to a trivial cocycle, with integrable transfer function.  
\end{enumerate}
\end{remark}

\end{document}